\documentclass[times, 11pt, doublespace]{amsart}
\usepackage{amsmath, amssymb, amsthm,verbatim}
\usepackage{mathtools}

\numberwithin{equation}{section}

\usepackage{setspace}
\usepackage[parfill]{parskip}

\usepackage[margin=1.5 in]{geometry}
\usepackage{tikz-cd}

\theoremstyle{plain}
\newtheorem{theorem}[equation]{Theorem}

\newtheorem{lemma}[equation]{Lemma}
\newtheorem{cor}[equation]{Corollary}
\newtheorem{conj}[equation]{Conjecture}

\newcounter{intro}
\theoremstyle{definition}
\newtheorem{definition}[equation]{Definition}

\newtheorem{remark}[equation]{Remark}

\DeclareMathOperator{\Spec}{Spec}
\DeclareMathOperator{\Pic}{Pic}

\DeclareMathOperator{\Div}{Div}
\DeclareMathOperator{\Prin}{Pr}
\DeclareMathOperator{\Mod}{mod}
\DeclareMathOperator{\Int}{int}

\DeclareMathOperator{\Nef}{nef}

\renewcommand{\div}{\text{div}}

\newcommand{\R}{\mathbb{R}}
\newcommand{\Z}{\mathbb{Z}}

\newcommand{\Q}{\mathbb{Q}}

\newcommand{\cA}{\mathcal A}

\newcommand{\cAo}{\mathcal{A}^{\circ}}

\newcommand{\cB}{\mathcal{B}}

\newcommand{\cD}{\mathcal{D}}
\newcommand{\ocD}{\overline{\mathcal{D}}}
\newcommand{\cE}{\mathcal{E}}

\newcommand{\cX}{\mathcal{X}}
\newcommand{\cU}{\mathcal{U}}
\newcommand{\cV}{\mathcal{V}}
\newcommand{\cW}{\mathcal{W}}
\newcommand{\cY}{\mathcal{Y}}

\newcommand{\oM}{\overline{M}}

\newcommand{\cL}{\mathcal{L}}
\newcommand{\oL}{\overline L}
\newcommand{\ocL}{\overline{\mathcal{L}}}

\newcommand{\vh}{\mathfrak h}

\begin{document}

\title{Specialization of canonical heights on abelian varieties}

\author{Alexander Carney}

\date{\today}

\maketitle
\begin{abstract}
Given a family of abelian varieties over a quasiprojective smooth curve $T^0$ over a global field and a point $P$ on the generic fiber, we show that the N\'eron-Tate canonical height $h_{X_t}(P_t)$ of $P_t$ along each fiber is exactly equal to a Weil height $h_{\oM}(t)$ given by an adelic metrized line bundle $\oM$ on the unique smooth projective curve $T$ containing $T^0$. As a consequence, 
we show that a conjecture of Zhang on the finiteness of small-height specializations of $P$ is equivalent to $\oM$ being big.
%we give a condition for two points $P$ and $Q$ to specialize to torsion points at the exact same places. This generalizes results of Tate and Silverman, and more recently DeMarco-Mavraki for elliptic curves, and contributes results in arbitrary dimension to the works of Masser-Zannier on unlikely intersections.  

%Call and Silverman's specialization theorem relates the canonical height on each fiber to the canonical height on the generic fiber times the height of the parameter, showing that ${h_{f_t}(P_t)=h_f(P)h_T(t)+o(h_T(t))}$. 

%Here, using Yuan and Zhang's theory of adelic line bundles, we shrink the error term, showing that for a general family, there exists a line bundle $M$ on $T$ such that ${h_{f_t}(P_t)=h_M(t)+O(1)}$. When $X$ is a family of abelian varieties, the existence of a N\'eron model allows us to eliminate the error entirely, so that ${h_{f_t}(P_t)=h_{\ocM}(t)}$ for a metrized line bundle $\ocM$ on $T$.

%When $X$ is an abelian variety, we show that the error term can be eliminated, by producing a metrized line bundle $\ocM$ on $T$ such that ${h_{f_t}(P_t)=h_{\ocM}(t)}$. For a general family of polarized dynamical systems, we show that the error term can be reduced to $O(1)$.
\end{abstract}

\section{Introduction}

Let $K$ be a number field or a transcendence degree one function field of any characteristic, and let $T^0$ be a smooth, quasiprojective curve over $K$. Suppose we have an abelian scheme $A\to T^0$, so that every fiber $A_t$ is an abelian variety. By fixing a line bundle on $A$ which restricts to a symmetric and ample line bundle on each fiber, we may specify the N\'eron-Tate height $h_{A_t}$ on each fiber in a consistent way.

%Choose some line bundle $\cL$ on $X$ which restricts to a symmetric and ample line bundle $\cL_t$ on each fiber $X_t$, and write $h_{X_t}$ for the N\'eron-Tate canonical height that this specifies.

%A polarized dynamical system consists of a projective variety $X$, an endomorphism $f:X\to X$, and an ample line bundle $L$ on $X$ such that $f^*L=qL$ for some $q>1$. Call and Silverman~\cite{CS} show how the polarization $L$ can be used to define the canonical height, which is now one of the most important tools in arithmetic dynamics. In the special case that $X$ is an abelian variety and $f$ is scalar multiplication by two, this is the same as the N\'eron-Tate height.

%Here, following a line of study begun by Call, Silverman, and Tate, we consider an algebraic family of polarized dynamical systems, and will show how the canonical height varies in such a family. Let $K$ be a number field or a transcendence degree one function field, let $T$ be a smooth projective curve over $K$, and let $F=K(T)$ be its function field. Let $(X,f,L)$ be a polarized dynamical system over $F$. Denote by $h_f$ the $f$-canonical height on $X(\overline F)$. If we choose a model for $X$ over $T$ (since all agree on an open set, the choice won't matter), the reduction $(X_t,f_t,L_t)$ is a polarized dynamical system which defines a canonical height $h_{f_t}$ for all but finitely many $t$. 

%We consider a family $(X,f,L)$ of polarized dynamical systems over $T$. Denote by $h_f$ the $f$-canonical height on $X(\overline F)$ and by $h_{f_t}$ the $f_t$-canonical height on all but finitely many fibers $X_t$. 

Fix a section $P:T^0\to A$ with specializations $P_t$. If we let $F=K(T^0)$ be the function field of $T$, this corresponds to an $F$ point on the generic fiber $A_{\eta}$.
Following a line of study begun by Silverman and Tate, we ask how the canonical height $h_{A_t}(P_t)$ varies as we vary the parameter $t$ along $T^0(\overline K)$. Let $T$ be the unique smooth projective curve containing $T^0$, and let $h_T(t)$ be any height on $T$ corresponding to a degree one line bundle on $T$. The general formulation of several theorems and conjectures is that 
\[
h_{A_t}(P_t)=h_A(P)h_T(t)+(\text{error term}),
\]
Where $h_A$ is the N\'eron-Tate canonical height on $A_{\eta}$.

Silverman conjectured~\cite{SilvermanThesis}, and Tate proved~\cite[Main Theorem]{Tate83} that if $A$ is a family of elliptic curves, a divisor $D_P$ can be found on $T$ so that replacing $h_T$ with $h_{D_P}$ reduces the error term to $O(1)$. Call~\cite{Call87} and~\cite{Green89} prove similar results in higher dimension. A series of subsequent results of Silverman culminating in~\cite{Silverman94} further classify the nature of the bounded error function for families of elliptic curves, and work of Biesel, Holmes, and de Jong~\cite{deJong} makes similar refinements in higher dimension. More recently, DeMarco and Mavraki \cite[Theorem 1.1]{DeMM} showed how to eliminate the error term entirely for families of elliptic curves by replacing $h_T$ with a specific height constructed using adelic metrized line bundles. Here, we generalize their result to any dimension, resolving this question in full for abelian varieties.

%Write $A$ for an abelian variety in place of $X$, and $h_A$ for the the N\'eron-Tate height.% by $h_A$.
\begin{theorem}\label{abelianvarieties}
Let $P\in A(F)$ such that $h_{A}(P)\ne 0$. The function
$$t\mapsto h_{A_t}(P_t),$$
defined for all but finitely many $t\in T$, can be extended to a height function $h_{\oM}$ on all of $T$ which is given by a nef adelic line bundle $\oM$ on $T$ whose underlying line bundle $M$ is ample. 
\end{theorem}

This is proved by a completely different method from that of DeMarco and Mavraki. We use Yuan-Zhang's theory of adelic line bundles and vector heights~\cite{yz21} to construct a canonical adelic line bundle, essentially extending Tate's limiting method from $\R$-valued height functions to the underlying geometry. Once this theory is established, the existence of a N\'eron model for $A$ provides the needed global geometric setting on which to complete the argument.

\begin{remark}
By the Lang-N\'eron Theorem (see~\cite{conradtrace}), if $A_{\eta}$ has trivial $F/K$-trace, the condition that $h_A(P)\ne0$ is equivalent to $P$ being non-torsion. In general, it's equivalent to $P$ being outside a torsion coset of Tr$_{F/K}(A_{\eta})$. 
\end{remark}

With a metrized line bundle giving $h_{A_t}(P_t)$ exactly, one could use equidistribution arguments to show, for example, that if $P,Q\in A(F)$ are two points whose heights are both non-zero and there exists a sequence $\{t_n\}\subset T(\overline \Q)$ such that both $h_{A_{t_n}}(P_{t_n})$ and $h_{A_{t_n}}(Q_{t_n})$ tend to zero as $n\to\infty$, then $P$ and $Q$ specialize to torsion points at the exact same places. This is likely vacuous, however, due to the following conjecture of Zhang.

\begin{conj}(\cite{ZhangConj})
Suppose $A\to T^0$ is a non-isotrivial family of abelian varieties as above, with $K$ a number field, and that $A_{\eta}$ is simple of dimension at least two. For each non-torsion section $P:T^0\to A$ defined over $\overline\Q$, there exists an $\epsilon>0$ such that 
\[
\{t\in T^0(\overline\Q): h_{A_t}(P_t)\le\epsilon\}
\]
is finite. 
\end{conj}

Instead, using Theorem~\ref{abelianvarieties}, we can restate Zhang's conjecture as a positivity statement for self-intersections.

\begin{cor}\label{Bigness}
Let $\oM$ be the metrized line bundle defined in Theorem~\ref{abelianvarieties} so that $h_{A_t}(P_t)=h_{\oM}(t)$. Then the conclusion of Zhang's conjecture is equivalent to 
\[
\oM^2>0.
\]
%Note that, since $\cM$ is nef, we have that $\cM^2\ge0$.
\end{cor}

The author is extremely grateful to Xinyi Yuan, Laura DeMarco, Niki Myrto Mavraki, and Tom Tucker for several ideas and insightful discussion while preparing this paper.
%, and reduce the error for general $X$ to $O(h_T(t)^{1/2})$, which is the best possible error without restricting the choice of height $h_T$. 

%Fix a number field $K$. Let $B$ be a smooth, projective curve over $K$, and let $F=K(B)$ be the function field of $B$. Let $A\to B$ be a projective, flat morphism, whose generic fiber $A_F\to\Spec F$ is an abelian variety. We denote by $h_{A}$ the N\'eron-Tate canonical height on $A_F$, and for every place $t\in B(\overline K)$ over which $A$ has  good reduction, write $h_{A_t}$ for the N\'eron-Tate canonical height on the fiber $A_t$.

%Mention how this is equivalent to metrized line bundles, and how the general version for finitely generated fields is proven later.

\section{Vector-valued heights}

We begin with an overview of adelic line bundles and vector-valued height functions, in so much as is needed here. For a more complete treatment, see~\cite[Ch. 2]{yz21} and ~\cite[Sec. 2]{Carney2} for more development of the function field setting. 

We will work in two parallel settings The first is for varieties over a field, in which all constructions are purely geometric, and which we will call the \emph{geometric case}. The second is for varieties over $\Spec \Z$, which we will call the \emph{arithmetic case}. Note that when $K$ is a number field, we will require both settings, as the height on $A_{\eta}$ is geometric, while that on each fiber is arithmetic.

\textbf{Geometric case:} Fix any field $k$, and let $\cU$ be a quasi-projective $k$ variety. A \emph{projective model} for $\cU$ is a projective variety $\cX$ along with an open immersion $\cU\hookrightarrow\cX$ defined over $k$. We write $\Div(\cX)$ for the group of Cartier divisors on $\cX$ and $\Prin(\cX)$ for the subgroup of principal Cartier divisors.

\textbf{Arithmetic case:} Let $\cU\to\Spec \Z$ be a quasi-projective arithmetic variety. In this setting, a \emph{projective model} for $\cU$ is a projective arithmetic variety $\cX\to\Spec\Z$ with an open embedding $\cU\to\cX$ over $\Spec\Z$. Write $\widehat\Div(\cX)$ for the group of arithmetic divisors on $\cX$, and $\widehat\Prin(\cX)$ for the subgroup of principal arithmetic divisors. 

In both cases the projective models for $\cU$ form an inverse system via dominating morphisms. Using pullbacks on that system, we define the following, first in the geometric case.
If $\cU\to\Spec k$ is a quasi-projective variety, 
$$\widehat\Div(\cU/k)_{\Mod}:=\lim_{\cU\hookrightarrow\cX}\Div(\cX)_{\Q}\quad\text{and}\quad\widehat\Prin(\cU/k)_{\Mod}:=\lim_{\cU\hookrightarrow\cX}\Prin(\cX)_{\Q}.$$
An element of $\widehat\Div(\cU/k)_{\Mod}$ is called \emph{effective} if it comes from an effective divisor in some $\Div(\cX)_\Q$. 
If $\cU\to\Spec\Z$ is a quasi-projective arithmetic variety, we define
$$\widehat\Div(\cU/\Z)_{\Mod}:=\lim_{\cU\hookrightarrow\cX}\widehat\Div(\cX)_{\Q}\quad\text{and}\quad
\widehat\Prin(\cU/\Z)_{\Mod}:=\lim_{\cU\hookrightarrow\cX}\widehat\Prin(\cX)_{\Q}.$$
An element of $\widehat\Div(\cU/\Z)_{\Mod}$ is called \emph{effective} if it comes from an effective arithmetic divisor in some $\widehat\Div(\cX)_\Q$, meaning that the finite part of the divisor is effective, and the Green's function is non-negative.

Next we define a topology on both of these groups, stemming from effectivity. Since the argument is identical in both the arithmetic and geometric cases we provide it only for the arithmetic case.

Let $\cU$ be a quasi-projective arithmetic variety. Effectivity provides a partial ordering on $\widehat\Div(\cU/\Z)_{\Mod}$. Fix some projective model $\cX_0$ and a strictly effective arithmetic divisor $\overline\cD_0$ on $\cX_0$ such that the support of $\cD_0$ is equal to $\cX_0\backslash\cU$. We call such $\cD_0$ a \emph{boundary divisor}. For $\epsilon\in\Q_{>0}$, define a basis of epsilon balls around $0$ by
$$B(\epsilon,0):=\left\{\overline\cE\in\widehat\Div(\cU/\Z)_{\Mod}:\epsilon\overline\cD_0\pm\overline\cE\text{ are both effective}\right\}.$$
%This is easily seen to be independent of the choice of $\overline\cD_0$. 
Via translation, this defines a topology on all of $\widehat\Div(\cU)_{\Mod}$, and it is easy to check that this topology does not depend on the choice of $\ocD_0$.

We now define $\widehat\Div(\cU/\Z)$ to be the completion of $\widehat\Div(\cU/\Z)_{\Mod}$ with respect to this topology. $\widehat\Div(\cU/k)$ is defined identically, except without a Green's function attached to each divisor. Finally, define the groups of \emph{adelic line bundles} in each setting:
$$\widehat\Pic(\cU/\Z):=\widehat\Div(\cU/\Z)/\widehat\Prin(\cU/\Z)_{\Mod},$$
$$\widehat\Pic(\cU/k):=\widehat\Div(\cU/k)/\widehat\Prin(\cU/k)_{\Mod}.$$

To further justify notating these as Picard groups, elements of $\widehat\Pic(\cU/k)$ (resp. $\widehat\Pic(\cU/\Z)$) can be represented by sequences $\{\cX_i,\psi_i,\cL_i,\ell_i\}_{i\ge1}$ (resp. $\{\cX_i,\psi_i,\ocL_i,\ell_i\}_{i\ge1}$) where $\cX_i$ is a projective model for $\cU$ with a morphism $\psi_i:\cX_i\to\cX_1$, where $\cL_i$ (resp. $\ocL_i$) is a $\Q$-line bundle (resp. Hermitian $\Q$-line bundle) on $\cX$, and $\ell_i$ is a rational section of $\cL_i\otimes\psi^*\cL_1^{-1}$ with support contained in $\cX_i\backslash\cU$. %In particular, we require $\ell_1=1$. 
The equality of this representation with the definitions above is shown in~\cite[Lemma 2.5.1]{yz21}.

We will typically just write $\ocL=\{\cX_i,\cL_i\}\in\widehat\Pic(\cU/k)$ or $\ocL=\{\cX_i,\ocL_i\}\in\widehat\Pic(\cU/\Z)$. From the conditions on $\ell_i$, we get a well-defined restriction map $\ocL\mapsto\cL\in\Pic(\cU)_{\Q}$.

Such a sequence is \emph{Cauchy} provided that $\{\div(\ell_i)\}$ (resp. $\{\widehat\div(\ell_i)\}$) is Cauchy under the topology defined above. A sequence converges to zero if there exists a sequence of rational sections $s_i$ of $\cL_i$ such that $\ell_i=s_i\otimes\psi_i^*s_1^{-1}$, and such that $\{\div(s_i)\}$ (resp. $\{\widehat\div(s_i)\}$) is itself Cauchy.

We call an adelic line bundle \emph{nef} if it is isomorphic to a sequence where every (Hermitian) line bundle is nef, and we call an adelic line bundle \emph{integrable} if it can be written as the difference of two nef ones. Denote the cones of nef elements
\[
\widehat\Pic(\cU/k)_{\Nef} \quad\text{ and }\quad \widehat\Pic(\cU/\Z)_{\Nef}
\]
and the subgroups of integrable elements
\[
\widehat\Pic(\cU/k)_{\Int} \quad \text{ and }\quad \widehat\Pic(\cU/\Z)_{\Int}.
\]

\subsection{The relative setting}
From here on out, we unify the notation for the base $k$ and base $\Z$ settings. Let $b$ be either $\Spec\Z$ or $\Spec k$. Let $K$ be a field which is finitely generated over $\Q$, if $b=\Spec\Z$, or finitely generated over $k$, when $b=\Spec k$. 
%Let $k$ be any field, and let $K$ be a finitely generated field extension of $k$. 
%Let $X$ be a quasi-projective variety over $K$. 
An \emph{open model} for $K$ is a quasi-projective $b$-variety $\cV$ with function field $K$. %When $k=\Q$, we also define an \emph{open arithmetic model} for $K$ to be a quasi-projective arithmetic variety $\cV\to\Spec\Z$ whose function field is $K$.

Next, let $X$ be a quasi-projective variety over $K$. An open model for $X/K$ consists of an open model $\cV$ for $K$, together with a quasi-projective and flat morphism $\cU\to\cV$ whose generic fiber is $X\to \Spec K$.

The open models for $X$ form an inverse system via inclusion. Taking the limit over this system, we define
%$$\widetilde\Pic(X/k):=\lim_{\cU\to\cV}\widetilde\Pic(\cU/k)_{\Cont},\text{ and}$$
\[
\widehat\Pic(X/b):=\lim_{\cU\to\cV}\widehat\Pic(\cU/b).
\]
When $X=\Spec K$, we will simply write $\widehat\Pic(K/b)$. Restricting these limits to the nef and integrable elements, we define the nef cone and integrable subgroup
%$$\widetilde\Pic(X/k)_{\Nef}\subset\widetilde\Pic(X/k)_{\Int}\subset\widetilde\Pic(X/k)_{\Cont},\text{ and}$$
\[
\widehat\Pic(X/b)_{\Nef}\subset\widehat\Pic(X/b)_{\Int}\subset\widehat\Pic(X/b).
\]

We will write elements as $\oL\in\widehat\Pic(X/b)$. 
%Given such $\oL$, we will write simply $L\in\Pic(X)_{\Q}$ for the uniquely defined $\Q$-line bundle which is the generic fiber of any line bundle in the limit defining $\oL$.
Given such $\oL=\{\cX_i,\ocL_i\}_{i\ge0}\in\widehat\Pic(X/b)$, we can restrict via base change to
$$\oL_K:=\{\cX_{i,K},\cL_{i,K}\}\in\widehat\Pic(X/K),$$
as each $\cX_{i,K}$ is a projective model for $X$. Since, by construction, all $\cL_{i,K}$ agree on $X\subset\cX_{i,K}$, we have a well defined restriction to $X$ which we denote simply $L\in Pic(X)_{\Q}$.

If $X^0$ is a smooth, quasi-projective $K$ curve, since $X^0$ has a unique projective model $X$ over $K$, 
\begin{equation}\label{curverestriction}
\widehat\Pic(X^0/K)=\Pic(X)_{\R}:=\Pic(X)\otimes_{\Z}\R.
\end{equation}

When $K$ is a number field or a function field of transcendence degree one over $k$ and $X$ is a projective $K$-variety, $\widehat\Pic(X/b)$ corresponds to the group of adelic metrized line bundles, as defined by Zhang~\cite{Z95}. There also exists an analytic theory of adelic line bundles over larger fields, for example~\cite[Sec. 3]{yz21},~\cite[Sec. 2]{Carney2}, but we will not require it here.

%{\color{red} Add restriction map here, give an equation number.}

%$$\widehat\Pic(X/b)_{\Int}\to\widehat\Pic(X/K)_{\Int}$$.

\subsection{Pullbacks}
Given a map $f:Y\to X$, since a priori the models defining $\widehat\Pic(Y/b)$ and $\widehat\Pic(X/b)$ may not be compatible, it is not immediately apparent that we can define pullbacks of adelic line bundles. We rectify that.

\begin{lemma}\label{pullbacklemma}
(C.f.~\cite[Sec. 2.5.5]{yz21})
Let $f:Y\to X$ a morphism of quasi-projective varieties which is flat over $b$. %either $k$ or $\Z$. 
Then there exists a well-defined pullback morphism
%$$f^*:\widetilde\Pic(X/k)_{\Cont}\to\widetilde\Pic(Y/k)_{\Cont},\text{ or}$$
$$f^*:\widehat\Pic(X/b)\to\widehat\Pic(Y/b).$$
\end{lemma}

\begin{proof}
%For simplicity, we prove this just in the base $k$ case. The proof in base $\Z$ is identical. 
%The reason this is not immediate is that, a priori, the limits defining these two groups may not be compatible. We rectify that. 
Assume for simplicity of notation that $b=\Spec\Z$; the only difference when $b=\Spec k$ is the lack of Hermitian metrics. Let $\oL\in\widehat\Pic(X/b)$ be represented by a Cauchy sequence $\{\cX_i,\ocL_i\}$.

By moving to smaller open subvarieties as needed, we may extend $f$ to a morphism $f:\cW\to\cU$ of open models for $X$ and $Y$, and it suffices to prove that there exists a well defined pullback $\widehat\Pic(\cU/b)\to\widehat\Pic(\cW/b)$.

Let $\cY$ be a projective model for $\cW$.
By Raynaud's flattening theorem~\cite{raynaudflattening}, by blowing up along $\cY\backslash\cW$, we may assume that $f:\cW\to\cX_i$ extends to a flat morphism $f_i:\cY_i\to\cX_i$. Replacing $\cY_i$ with a morphism which dominates it, we may further assume that there is a morphism $\psi_i:\cY_i\to\cY_0$ for every $i$, which is the identity on $\cW$.

We now have a compatible system of models on which to pull back $\oL$. Define 
\[
f^*\oL:=\left\{\cY_i,\psi_i,f_i^*\ocL_i,f^*\ell_i\right\}_{i\ge0}.
\]
Since effective Cartier divisors pull back to effective Cartier divisors under both flat and dominant maps, and the topology on $\widehat\Pic(\cW/b)$ does not depend on the choice of boundary divisor, the topology on $\widehat\Pic(\cU/b)$ induces that on $\widehat\Pic(\cW/b)$ via pullback of any boundary divisor on $\cX_0$, and the sequence defining $f^*\oL$ is Cauchy with respect to this topology.

\end{proof}

%{\color{red}Check that pullbacks of nef are nef, probably Lazarsfeld positivity}
%Since we will use this fact later, note that the pullback of a nef adelic line bundle is also nef.
%(False?)

\section{Canonical heights}

\begin{definition}
Let $X$ be a quasi-projective $K$ variety, with $K$ and $b$ as in the previous section. Let $\oL\in\widehat\Pic(X/b)_{\Int}$, and suppose that $L\in\Pic(X)$ is ample. Let $P\in X(K)$. We define the \emph{vector valued $\oL$-height} of $P$ to be
$$\vh_{\oL}(P):=P^*\oL\in\widehat\Pic(K/b)_{\Int}.$$
If $P$ is defined over some finite extension $K'$, we compose with the map $\widehat\Pic(K'/b)_{\Nef}\to\widehat\Pic(K/b)_{\Nef}$ induced by the norm functor $N_{K'/K}$ and scale the result by $1/[K':K]$, so that this is a relative height which does not depend on the choice of $K'$. For simplicity, we will usually omit this formalism and assume without loss of generality that $P$ is defined over $K$.
%If we are in the base $k$ setting, we similarly define $h_{\oL}(P)\in\widetilde\Pic(K/k)_{\Nef}$.
\end{definition}

When $K$ is a number field or a transcendence degree one function field over $k$, taking the limit of the arithmetic degrees produces a Weil height~\cite{Z95}, which we notate
\[
h_{\oL}(P):=\widehat\deg \left(\vh_{\oL}(P)\right),
\]
and thus vector-valued heights extend the theory of $\R$-valued heights. 

%\begin{remark}\label{largerfields}
%For higher transcendence degree extensions of $k$ (resp. $\Q$), one can still define $\R$-valued Weil heights (resp. Moriwaki heights), provided one first fixes a \emph{polarization}, namely a model for $K$ over $b$, and a choice of (Hermitian) line bundles on that model. Note that this terminology is unrelated to the notion of a polarization of an algebraic dynamical system. In the language of adelic line bundles, one obtains a Weil height (resp. Moriwaki height) over a larger transcendence degree field $K$ by intersecting $\vh_{\oL}(P)$ with a choice of $\oH_1,\dotsm\oH_d\in\widehat\Pic(K/b)_{\Nef}$, where ${d+1=\trdeg(K/k)}$ (resp. $d=\trdeg(K/\Q)$). The intersection theory of adelic line bundles can be found in~\cite{yz21,Carney2}, but we won't need it here.
%\end{remark}

%{\color{red}Note on degree map and $\R$-valued heights here.}
%When $K$ is a height field, this can be easily related to the more familiar $\R$-valued heights, (or Moriwaki heights). Cite Silverman.

\begin{lemma}\label{tateslimiting}
(C.f.~\cite[Sec. 6.1.1]{yz21},~\cite[Sec. 4.1]{Carney2})
Let $f:X\to X$ be an endomorphism of projective varieties over $K$ with a \emph{polarization}, i.e. an ample $\Q$-line bundle $L\in\Pic(X)_{\Q}$ and a rational number $q>1$ such that $f^*L=qL$. Then there exists a nef adelic line bundle $\oL_f\in\widehat\Pic(X/b)_{\Nef}$ extending $L$ such that 
$$f^*\oL_f=q\oL_f.$$
In particular, there exists a canonical height $\vh_f:=\vh_{\oL_f}$ such that for all $P\in X(\overline K)$,
$$\vh_f(f(P))=q\vh_f(P).$$
\end{lemma}

\begin{proof}
We again prove this only in the base $\Z$ case, as the base $k$ case is identical besides the lack of Hermitian metrics. To start, we can find an open model $\cU\to\cV$ over $b$ for $X\to\Spec K$ on which $f$ extends to a $\cV$-morphism $f:\cU\to\cU$, and over which there exists an extension $\cL\in\Pic(\cU)_{\Q}$ of $L$ such that $f^*\cL=q\cL$. Fix a projective model $\cX_0\to \cB$ for $\cU\to\cV$, and a nef extension $\ocL_0\in\widehat\Pic(\cX_0)_{\Q}$ of $\cL$. By shrinking $\cV$ if needed, we may assume $\cB\backslash\cV$ is an effective Cartier divisor.

%We define $\cD_0$ to be an effective Cartier divisor whose support is $\cX_0\backslash\cU$.

Write $f^m$ to mean the $m$-th iterate of $f$. For each $m\ge1$, define $f_m:\cX_m\to\cX_0$ to be the normalization of the composition $X\xrightarrow{f^m}X\hookrightarrow\cX_0$, and define 
$$\ocL_m:=\frac1{q^m}f_m^*\ocL_0.$$
%Note that, since $L$ is ample, each $\ocL_m$ is nef. 

For each $m\ge0$, let $\cX_m'$ be a projective model for $\cU$ which dominates both $\cX_m$ and $\cX_{m+1}$ via $\phi_m:\cX_m'\to\cX_m$ and $\rho_m:\cX_m'\to\cX_{m+1}$, both of which restrict to the identity on $\cU$. Further assume that for $m\ge1$, there exists a morphism $\tau_m:\cX_m'\to\cX_0'$ which commutes with $\phi_m,$ $\rho_m$ and $f_m$. By construction $\tau_m$ extends $f^m$ on $\cU$. 
Since $\phi_0^*\ocL_0-\rho_0^*\ocL_1$ is trivial on $\cU$, we can find a boundary divisor $\ocD_0$ supported on $\cB\backslash\cV$ such that for the topology defined by $\pi^*\ocD_0$, we have $\phi_0^*\ocL_0-\rho_0^*\ocL_1\in B(1,0)$. Now
\begin{equation*}\label{tateslimitingequation}
\phi_m^*\ocL_m-\rho_m^*\ocL_{m+1}=q^{-m}\tau_m^*\left(\phi_0^*\ocL_0-\rho_0^*\ocL_1\right).
\end{equation*}
and thus
\[
\phi_m^*\ocL_m-\rho_m^*\ocL_{m+1}\in B\left(q^{-m},0\right).
\]
We get a Cauchy sequence
\[
\oL_f:=\left\{\cX_m,\ocL_m\right\}_{m\ge0}\in\widehat\Pic(X/b)_{\Nef},
\]
which is easily seen to satisfy the desired conditions.
%Since $\ocL_0$, and then every subsequent $\ocL_n$ is nef, $\oL_f$ is nef. That this satisfies the conclusion of the lemma, and allows us to define the canonical height, follows immediately.
\end{proof}

%\begin{Corollary}
%{\color{red}Connection to $\R$-valued canonical heights when $K$ is a number field or a transcendence degree one extension of $k$.}
%\end{cor}

\section{Proof of Theorem~\ref{abelianvarieties}}

%We state and prove a more general theorem which includes function fields as well as number fields. 

Let $K$ be a number field and $b=\Spec\Z$, or let $K$ be a transcendence degree one function field with constant field $k$, and $b=\Spec k$. Let $T^0$ be a smooth quasiprojective curve over $K$, and let $F=K(T^0)$ be its function field. Let $A$ be a family of abelian varieties over $T^0$, and let $T$ be the unique smooth projective curve with function field $F$. 

Write $\vh_{A}=\vh_{[2]}$ for the $\widehat\Pic(F/K)_{\Int}$-valued canonical height on the generic fiber $A_{\eta}$, and on each fiber $A_t$ over $t\in T^0(\overline K)$, write $\vh_{A_t}$ for the $\widehat\Pic(K/b)_{\Int}-$valued canonical height on $A_t$. As defined in the previous section, the usual $\R$-valued N\'eron-Tate canonical heights are 
$$h_A(P)=\widehat\deg\left(\vh_A(P)\right),\qquad h_{A_t}(P_t)=\widehat\deg\left(\vh_{A_t}(P_t)\right)$$

%\begin{theorem}\label{generaltheorem}
We show that if $h_{A}(P)\ne0$, there exists a nef adelic line bundle $\oM\in\widehat\Pic(K/b)_{\Nef}$ such that on the $\overline K$-points of $T$ over which $A$ has good reduction, the function 
$$t\mapsto \vh_{A_t}(P_t)$$
is equal to a vector-valued height function $\vh_{\oM}$, defined on all of $T$.
%\end{theorem}

%In the special case that $K$ is a number field, or $K$ is a transcendence degree one function field over $k$, taking the arithmetic degree of $\vh_{A_t}(P_t)$ produces an $\R$-valued Weil height, proving Theorem~\ref{abelianvarieties}.

%Recall the setting. $k$ is any field, $K$ is finitely generated over $k$, and $B\to\Spec K$ is a smooth projective curve with function field $F$. Over $B$, we have a family $A\to B$ of abelian varieties, with generic fiber $A_F\to\Spec F$. Fix a point $P\in A(F)$ with non-zero canonical height.

%Having established the language of vector-valued heights, all that's needed is a way to extend the morphism $[2]$ from the smooth fibers of $A$ to fibers over all of $B$.

\begin{proof}
Let $\cA\to T$ be the N\'eron model for $A_{\eta}$. This is a quasiprojective $K$-variety. Replacing $F$, and thus $T$, with a finite extension will not alter the result, as the heights involved are invariant under finite extensions. Thus, we may assume that $\cA\to T$ has semiabelian reduction~\cite[Ch. 7.4, Thm 1]{RaynaudBook}. We also write $P:T\to\cA$ for the unique section extending $P\in A(F)$, by the N\'eron mapping property.

Since $\cA$ has abelian reduction at all but finitely many fibers, and the bad fibers have finitely many components, by replacing the point $P$ with $nP$ if needed, we may assume that the section $P$ is contained in the connected component of the identity, which we denote $\cA^{\circ}$. This replacement scales the canonical height on both $A$ and on each abelian fiber $A_t=\cA_t$ by $n^2$, and thus does not affect the result.

Let $\cX\to b$ be a projective model for $\cAo$ over $b$. Fix a symmetric and ample line bundle $L$ on $A_{\eta}$, and extend this to a nef line bundle $\cL$ on $\cX$. By applying the normalization construction from the proof of Theorem~\ref{tateslimiting}, we can find a second projective model $\cX'\to b$ with a morphism
\[
[-1]:\cX'\to\cX
\]
extending the automorphism $[-1]$ on $\cAo$. Fix a third model $\cX_0$ dominating both $\cX$ and $\cX'$. We can then define a nef line bundle
\[
\cL_0=\frac12\left(\cL+[-1]^*\cL\right)\in\Pic(\cX_0)_{\Q},
\]
where we use the implied pullbacks from $\cX$ and $\cX'$ to $\cX_0$.

We adapt an argument of Green~\cite{Green89}. Consider the difference $[n]^*\cL_0|_{\cAo}-n^2\cL_0|_{\cAo}$. By the theorem of the cube, this becomes trivial when restricted to $A_{\eta}=\cA_{\eta}$. Then, since $\cAo$ is smooth and has integral fibers over $T$, this difference is equal to the pullback under $\cAo\to T$ of a line bundle on $T$, by~\cite[Cor. 21.4.13 of Ch. 4, Errata et Addenda]{EGA4}. Finally, pulling back by the identity section $E:T\to\cAo$, we see that $[n]^*\cL_0|_{\cAo}-n^2\cL_0|_{\cAo}$ must be trivial. Thus, 
\[
[n]^*\cL_0|_{\cAo}=n^2\cL_0|_{\cAo}.
\]
\begin{remark}
For a different version of the above argument extending a polarization to $\cAo$, one can use Moret-Bailly's extension of cubical structure,~\cite[Thm. 3.5]{MoretBailly}, as in~\cite[Sec. 4]{Kunnemann}.
\end{remark}

%Next, following~\cite{SerreGroups} and~\cite[Section 2]{VojtaI}, we can complete each bad fiber $\cA_t^{\circ}$, which is semiabelian, to a proper variety, and since properness is local on the base, we have a proper map $\overline{\cA^{\circ}}\to T$ into which $\cA^{\circ}$ embeds as an open subvariety. While the full group multiplication morphism doesn't extend, scalar multiplication $[n]:\overline{\cA^{\circ}}\to\overline{\cA^{\circ}}$ does~\cite[Prop. 2]{SerreGroups}.

%Let $L$ be a symmetric and ample line bundle on $A$ corresponding to the $\R$-valued N\'eron-Tate height, and so that $[2]^*L=4L$. Fix any $\Q$-line bundle $\cL_0$ on $\overline{\cA^{\circ}}$ extending $L$, in the sense that $\cL_{0,F}=L$. Since $\overline{\cA^{\circ}}$ is proper and $[2]$ extends, the theorem of the cube holds. We can thus define a symmetric $\Q$-line bundle 
%$$\cL=\frac12\left(\cL_0+[-1]^*\cL\right),$$
%for which $\cL_F=L$, and $[2]^*\cL=4\cL$.

%{\color{red} Should we work with $\ocA$ or pull this back to $\cA$?}

%Write $\cL$ for the restriction to $\cA^{\circ}$ as well. 
Next, apply Lemma~\ref{tateslimiting} to extend $L$ to a nef adelic line bundle $\oL=\{\cX_m,\ocL_m\}\in\widehat\Pic(A/b)_{\Nef}$ such that
$$[2]^*\oL=4\oL.$$
In general, the model adelic line bundles $\ocL_m$ may all differ along a boundary divisor. Using the above construction extending the polarization to $\cAo$, however, the line bundles $\cL_{m,K}$ can be made to all agree on $\cAo$. Since the section $P$ lands fully within $\cAo$, we have a well-defined pullback
\[
\oM=P^*\oL\in\widehat\Pic(T/b)_{\Int},
\]
since each $P^*(\cL_{m,K})$ is the same line bundle on $T$.

By construction, $\vh_{\oM}(t)=t^*\oM=P_t^*\oL_t=\vh_{A_t}(P_t)$ for all $t\in T^0(\overline K)$. Since $T$ is a curve over $K$, the $\R$-valued N\'eron-Tate height of $P$ is simply $\deg(M)=\deg((P^*\oL)_K)\in\Q$. Since $P$ does not have height zero, we conclude that $M$ is ample, and that $\oM$ is nef, as the pullback of a nef adelic line bundle.
\end{proof}

\begin{remark}
Without the N\'eron model, which exists for abelian varieties, but not in general for dynamical systems (see for example~\cite{Hsia96}), the above construction would instead produce only $\oM\in\widehat\Pic(T^0/b)$. Since $T^0$ has a unique projective closure, this gives an $\R$-line bundle $M\in\Pic(T)_{\R}$, but it remains an open question whether the adelic structure extends over all of $T$, i.e. whether the metrics defined by each $\ocL_m$ converge to an adelic metric on $T$.

\end{remark}

\subsection{Proof of Corollary~\ref{Bigness}}

%Adopt the notation of the previous section, in particular that $\cM$ is the nef adelic line bundle on $T$ which gives the canonical height along each fiber. 
Since $\oM$ is nef, we know that $\oM^2\ge0$. Define
\[
e_1:=\sup_{\text{open }U\subset T}\inf_{t\in U(\overline K)}h_{\oM}(t).
\]
It can easily be seen that $e_1\ge0$, and is positive if and only if Zhang's conjecture holds. 

By Zhang's essential inequalities~\cite[Thm 5.2]{Zhang95a},
\[
e_1\ge\frac{\oM^2}{\deg M}\ge\frac{e_1}2.
\]
From this, the equivalence of the restatement of the conjecture in terms of the bigness of $\oM$ is immediate.

%\section{Variation of dynamical canonical height}

%Very easy, as above if $f:\cX\to\cX$ has a polarization. Unclear if possible if not; likely not.

%{\footnotesize
%\bibliographystyle{alpha}
%\bibliography{references}
%}

\end{document}